\documentclass[12pt, twoside, leqno]{article}

\usepackage{amsmath,amsthm}
\usepackage{amssymb}

\usepackage{enumerate}

\usepackage{graphicx}

\newtheorem{thm}{Theorem}[section]

\newtheorem{lem}[thm]{Lemma}
\newtheorem{prop}[thm]{Proposition}

\theoremstyle{definition}

\newtheorem{exa}[thm]{Example}

\numberwithin{equation}{section}

\begin{document}


\baselineskip=17pt


\title{On certain generalized isometries of the special orthogonal group}

\author{Marcell Ga\'al\\
Bolyai Institute, University of Szeged\\
H-6720 Szeged, Aradi v\'ertan\'uk tere 1., Hungary\\
and \\
MTA DE "Lend\"ulet" Functional Analysis Research Group\\ 
Institute of Mathematics, University of Debrecen \\
H-4010 Debrecen P.O. Box 12, Hungary \\
E-mail: marcell.gaal.91@gmail.com
}
\date{}

\maketitle


\renewcommand{\thefootnote}{}



\renewcommand{\thefootnote}{\arabic{footnote}}
\setcounter{footnote}{0}


\begin{abstract}
In this paper we explore the structure of certain generalized isometries of the special orthogonal group $SO(n)$ which are transformations that leave any member of a large class of generalized distance measures invariant.
\end{abstract}

\section{Introduction}
Let $SO(n)$ and $\mathbb{K}_n(\mathbb{R})$ denote the special orthogonal group and the associated Lie algebra consisting of the set of all skew-symmetric real matrices, respectively. The symbols $\|.\|$ and $\|.\|_F$ stand for the operator norm and the Frobenius norm, respectively. For any $X\in\mathbb{K}_4(\mathbb{R})$ we denote by $\widetilde{X}$ the matrix which is obtained from $X$ by interchanging its (1,4) and (2,3) entries, and interchanging the (4,1) and (3,2) entries, respectively. 
If $M$ is a set and $d:M\times M \to [0,+\infty[$ is a function satisfying $d(x,y)=0$ if and only if $x=y$, then $d$ is termed to be a generalized distance measure. Clearly, a generalized distance measure may not be a metric in the usual sense because we require neither the symmetry nor the triangle inequality.

In \cite{abe} the structure of isometries of $SO(n)$ with respect to the metric induced by any $c$-spectral norm were determined. At the end of that paper the authors proposed an open problem: how one can describe the form of isometries with respect to any unitary invariant norm? The main purpose of this paper is to solve the essential part of the problem.
More precisely, we consider those kinds of generalized distance measures on the manifold $SO(n)$ which are given of the form
\begin{equation} \label{distance}
d_{N,f}(A,B)=N(f(A^{-1}B))
\end{equation}
where $N$ is any unitary invariant norm and $f:\mathbb{T} \to \mathbb{C}$ is a bounded function satisfying
\begin{itemize}
\item[(f1)] $f(z)=0$ if and only if $z=1$;
\item[(f2)] $f$ is conformal in a neighbourhood of $1$;
\item[(f3)] $f$ is continuous on $\mathbb{T}\setminus \{-1\}$;
\item[(f4)] $f(-1)=\lim_{t\uparrow \pi} f(e^{\mathrm{i}t})$;
\item[(f5)] $0\neq \lim_{t\downarrow -\pi} f(e^{\mathrm{i}t})$;
\end{itemize}
and determine the structure of those maps (called generalized isometries) which preserve the above type generalized distance measures between the elements of $SO(n)$. 
Here, for any $A\in SO(n)$ the matrix $f(A)$ is defined via the usual Borel function calculus.  
If we take the logarithm function $f(z)=\log z$, we apply the convention $\log(-1)=\mathrm{i}\pi$.

Observe that our problem is significantly more general than the original one where only the function $f(z)=z-1$ appears. 
Unfortunately, if $n>3$ we are not able to give the complete description of the above type generalized isometries with respect to any unitary invariant norm, but we are with respect to any unitary invariant norm which is not a scalar multiple of the Frobenius norm.
Our results give us a substantial generalization of the beautiful result \cite[Theorem 1]{abe} and also include the characterization of geodesic distance isometries on $SO(3)$. As recent literature on investigations concerning the structures of isometries and certain generalized isometries on matrix classes (or, much more generally, on $C^{*}$-algebras or certain classes of von Neumann algebras) the reader is referred to \cite{uni, special, hatori, spectral, mazur}.

We remark that the conditions (f1)-(f4) concerning the numerical function $f$ came from the requirement that we want to cover the functions $z\mapsto z-1$ and $z\mapsto \log z$ which correspond to the cases of the norm distance and the geodesic distance, respectively (see the examples below).

\begin{exa}
If $f(z)=z-1$, then we have
\begin{equation*}
d_{N,f}(A,B)=N(A-B) 
\end{equation*}
which is just the usual norm distance.
\end{exa}

\begin{exa}
If $f(z)=\log z $
and $N(.)=\|.\|_F$, then we have
\begin{equation*}
d_{N,f}(A,B)=\| \log\left(A^{-1}B\right) \|_F
\end{equation*}
which is the geodesic distance.
\end{exa}

These generalized distance measures are commonly used in robotics, computer vision, computer graphics and in the medical sciences, the reader can consult e.g. \cite{metrics, moakher} for further details.

\section{Results}
Our first result concerning generalized distance measures discussed in the previous section reads as follows.

\begin{thm} \label{T:1}
Let $N(.)$ be a unitary invariant norm which is not a constant multiple of the Froebenius norm and assume that $f\colon \mathbb{T} \to \mathbb{C}$ is a bounded function satisfying (f1)-(f4). The map $\phi \colon SO(n) \to SO(n)$ is a generalized isometry with respect to the generalized distance measure $d_{N,f}(.,.)$, i.e., 
\begin{equation*} 
d_{N,f}(\phi(A),\phi(B))=d_{N,f}(A,B), \qquad A,B\in SO(n)
\end{equation*}
if and only if there exists an orthogonal matrix $Q\in O(n)$ such that $\phi$ is of one of the following forms:
\begin{itemize}
\item[(a)] $\phi(A)=\phi(I)QAQ^{-1}$ for all $A\in SO(n)$;
\item[(b)] $\phi(A)=\phi(I)QA^{-1}Q^{-1}$ for all $A\in SO(n)$;
\item[(c)] $n=4$ and $\phi(A)=\phi(I)Q\exp(\widetilde{X})Q^{-1}$ for all $A \in SO(4)$, where $X\in \mathbb{K}_4(\mathbb{R})$ such that $\exp(X)=A$;
\item[(d)] $n=4$ and  $\phi(A)=\phi(I)Q\exp(-\widetilde{X})Q^{-1}$ for all $A \in SO(4)$, where $X\in \mathbb{K}_4(\mathbb{R})$ such that $\exp(X)=A$.
\end{itemize}
\end{thm}

Our second theorem reads as follows.

\begin{thm} \label{T:2}
Let $N(.)$ be a constant multiple of the Froebenius norm and assume that $f\colon \mathbb{T} \to \mathbb{C}$ is a bounded function satisfying (f1)-(f4). The map $\phi \colon SO(3) \to SO(3)$ is a generalized isometry with respect to the generalized distance measure $d_{N,f}(.,.)$
if and only if there exists an orthogonal matrix $Q\in O(3)$ such that $\phi$ is of either of the form (a) or (b) in Theorem~\ref{T:1}.
\end{thm}

For the proof we need some more preliminaries.
Since $SO(n)$ is a compact connected Lie group the exponential map defined by
\[
\exp: \mathbb{K}_n(\mathbb{R})\to SO(n),\qquad X\mapsto \sum\limits_{k=0}^{\infty} \frac{X^k}{k!}
\]
is surjective. Moreover, due to the Lie group-Lie algebra correspondence we have that if $\gamma(t)$ is a one-parameter subgroup, i.e.,
\begin{equation} \label{subgroup}
\gamma(t+s)=\gamma(t)\cdot\gamma(s),\qquad t,s\in\mathbb{R},
\end{equation}
then there exists an $X\in\mathbb{K}_n(\mathbb{R})$, the generator of $\gamma(t)$, for which $\gamma(t)=\exp(tX)$.

It is apparent that if $X,Y\in\mathbb{K}_n(\mathbb{R})$, then we have a special orthogonal matrix $BCH(X,Y)$ such that $\exp(X)\exp(Y)=\exp\left(BCH(X,Y)\right)$. According to the famous Baker-Campbell-Hausdorff formula we have
\[
\begin{gathered}
BCH(X,Y)=X+Y+\frac{1}{2}\left(XY-YX\right)+\\
\frac{1}{12}\left(X^2Y+XY^2-2XYX+Y^2X+YX^2-2YXY  \right)+...
\end{gathered}
\]
for the first three terms of the series expansion of $BCH(X,Y)$. The following lemma appeared in \cite{abe}.

\begin{lem} \cite[Lemma 8.]{abe} \label{L:0}
For any $X,Y\in\mathbb{K}_4(\mathbb{R})$ the eigenvalues of $BCH(X,Y)$ and $BCH(\widetilde{X},\widetilde{Y})$ coincide.
\end{lem}

Next we recall the Youla-decomposition of skew-symmetric matrices. Clearly, any skew-symmetric matrix is normal and thus unitary diagonalizable. Since all the nonzero eigenvalues of a skew-symmetric matrix are located on the imaginary axis it cannot be diagonalized by a real orthogonal matrix. Nevertheless, if the eigenvalues of the matrix $X\in \mathbb{K}_n(\mathbb{R})$ are $\{\pm \mathrm{i}\lambda_1, \pm \mathrm{i}\lambda_2, \ldots \pm \mathrm{i}\lambda_r, 0, 0, \ldots 0\}$, then there is an orthogonal matrix $Q$ such that $X$ can be decomposed as $X=Q\Sigma Q^{-1}$ where
\[
\Sigma=\begin{pmatrix}
0 & \lambda_1 & 0 & 0 & \ldots & 0 & 0 & 0 & \ldots & 0  \\
-\lambda_1 & 0 & 0 & 0 & \ldots & 0 & 0 & 0 & \ldots & 0  \\
0 & 0 & 0 & \lambda_2 & \ldots & 0 & 0 & 0 & \ldots & 0  \\
0 & 0 & -\lambda_2 & 0 & \ldots & 0 & 0 & 0 & \ldots & 0  \\
\vdots & \vdots & \vdots & \vdots & \ddots & \vdots & \vdots & \vdots & \ldots & \vdots \\
0  & 0  & 0 & 0 & \ldots & 0 & \lambda_r & 0 & \ldots & 0  \\
0  & 0  & 0 & 0 & \ldots & -\lambda_r & 0 & 0 & \ldots & 0  \\
0  & 0  & 0 & 0 & \ldots & 0 & 0 & 0 & \ldots & 0  \\
\vdots & \vdots & \vdots & \vdots & \ldots & \vdots & \vdots & \vdots & \ldots & \vdots \\
0  & 0 & 0  & 0  & \ldots & 0 & 0 & 0 & \ldots & 0  \\
\end{pmatrix}
\]
The above decomposition is called the Youla-decomposition of $X$.

We also need the following notions. For any $A,B\in SO(n)$ the operation $(A,B)\mapsto ABA$ is called the Jordan triple product of $A$ and $B$ while the operation $(A,B)\to AB^{-1}A$ is said to be their inverted Jordan triple product. A map $\phi:SO(n)\to SO(n)$ which satisfy $\phi(ABA)=\phi(A)\phi(B)\phi(A)$ for every $A,B\in SO(n)$
is called a Jordan triple endomorphism.
Similarly, a map $\phi$ on $SO(n)$ fulfilling 
$\phi(AB^{-1}A)=\phi(A)\phi(B)^{-1}\phi(A)$ for all $A,B\in SO(n)$
is said to be an inverted Jordan triple endomorphism.

A map is called unital if it sends the unit to the unit. It is not difficult to verify that every unital Jordan triple map is compatible with the inverse and the power operations, i.e., 
$\phi\left(A^k\right) =\phi(A)^k$ $(k\in\mathbb{N})$ and
$ \phi\left(A^{-1}\right) =\phi(A)^{-1}$
holds for every $A$.

\section{Proofs}
The main steps of the proofs follow the ones that appeared in the arguments given in \cite{abe} and \cite{special}, however, the details are different at several points. 
We rely heavily on the following general Mazur-Ulam type result  which appeared in \cite{mazur}.

\begin{prop}\cite[Proposition 20]{mazur} \label{P:1}
Assume that $G$ and $H$ are groups equipped with generalized distance measures $d$ and $\rho$, respectively. Select $a,b\in G$ and set
\[
L_{a,b}:=\{ x\in G : d(a,x)=d(x,ba^{-1}b)=d(a,b) \},
\]
and assume the following:
\begin{itemize}
\item[(c1)]$d(bx^{-1}b,b{x'}^{-1}b)=d(x',x)$ holds for every $x,x'\in G$;
\item[(c2)]$\sup \{d(x,b): x\in L_{a,b} \}<\infty$;
\item[(c3)]there is a constant $K>1$ such that 
\[
d(x,bx^{-1}b)\geq Kd(x,b), \qquad x\in L_{a,b};
\]
\item[(c4)] $\rho(cy^{-1}c',c{y'}^{-1}c')=\rho(y',y)$ holds for every $y,y',c,c'\in H$.
\end{itemize}
Then for any surjective map $\phi\colon G\to H$ satisfying
\[
\rho(\phi(x),\phi(x'))=d(x,x')\qquad(x,x'\in G)
\]
we necessarily have
\[
\phi(ba^{-1}b)=\phi(b)\phi(a)^{-1}\phi(b).
\]
\end{prop}

We are now in a position to prove the following auxiliary lemma. 

\begin{lem} \label{L:1}
If $\phi: SO(n)\to SO(n)$ is a generalized isometry with respect to the generalized distance measure $d_{N,f}(.,.)$ which sends the unit to the unit, then $\phi$ is a continuous unital Jordan triple endomorphism with respect to the operator norm.
\end{lem}

\begin{proof}
We first show that $\phi$ is continuous with respect to the operator norm topology. Consider a fixed $A\in SO(n)$ and a sequence $(A_i)_{i\in \mathbb{N}} \in SO(n)$ such that $A_i \to A$ in the operator norm. Clearly, this implies that $A_i^{-1}A \to I$. By the continuity of $f$ on $\mathbb{T}\setminus\{-1\}$ and the property (f1), we infer that $f(A_i^{-1} A)\to 0$ in the operator norm. Since on a finite dimensional normed space every norm is complete and generates the same topology we have 
\[
N(f(A_i^{-1} A))=d_{N,f}(A_i,A)\to 0.
\]
Since $\phi$ preserves the generalized distance measure $d_{N,f}$ we also have
\[
N\left(f(\phi(A_i)^{-1}\phi(A))\right)=d_{N,f}(\phi(A_i),\phi(A))\to 0.
\]
It follows that $f(\phi(A_i)^{-1}\phi(A))\to 0$ in the operator norm. By the continuity of $f$ on $\mathbb{T}\setminus\{-1\}$ and the properties (f1), (f5) we necessarily have $\phi(A_i)^{-1}\phi(A)\to I$ implying that $\phi(A_i)\to\phi(A)$ in the operator norm.

Our aim is now to show that $\phi$ is surjective. There is a folk result (see e.g. \cite[Excercise 2.4.1]{folk}) saying that isometries from a compact metric space into itself are automatically surjective. Here we use the ideas from its proof and adjust them to the setting of generalized distance measures. So, assume for contradiction that $\phi$ is not surjective. Then there exists an $A\in SO(n)$ which is not contained in the image of $\phi$. This yields
\[
\inf \{ \|\phi(B)-A\| : B\in SO(n) \} > 0
\]
and, similarly to the above discussed argument relating to the continuity, we can infer that
\[
c:= \inf \{ d_{N,f}(\phi(B),A) : B\in SO(n) \} > 0.
\]
Set $A_0:=A$ and define the sequence $(A_i)_{i\in\mathbb{N}}\in SO(n)$ by the induction $A_{i}:=\phi(A_{i-1})$. Since $\phi$ is a generalized isometry we obtain that
\[
d_{N,f}(A_i,A_j)=d_{N,f}(\phi^i(A_0),\phi^j(A_0))=d_{N,f}(\phi^{i-j}(A),A)\geq c> 0
\]
holds for every $i,j\in\mathbb{N}$ with $i>j$. This implies that $\|A_i - A_j \|$ is bounded away from zero. We notice that $\phi(SO(n))$, as an image of the compact set $SO(n)$ under the continuous map $\phi$, is compact. Hence one can choose a convergent subsequence of $(A_i)_{i\in\mathbb{N}}$, and thus the quantity $\|A_i - A_j \|$ cannot be bounded away from zero. Therefore, $\phi$ is a surjective generalized isometry.

We now intend to prove that $\phi$ preserves the inverted Jordan triple product locally. In order to do so, it is sufficient to show that the conditions (c1)-(c4) appearing in Proposition~\ref{P:1} are satisfied in the following setting: $G=H=SO(n)$ and $d=\rho=d_{N,f}$.
As for the condition (c1), we calculate
\[
\begin{gathered}
f(B^{-1}X{X'}^{-1}B)=B^{-1}f(X{X'}^{-1})B=\\
B^{-1}f(({X'}{X'}^{-1})X{X'}^{-1})B=B^{-1}{X'}f({X'}^{-1}X){X'}^{-1}B.
\end{gathered}
\]
Since $N$ is unitary (orthogonal) invariant we infer from this that
\[
\begin{gathered}
d_{N,f}(BX^{-1}B,B{X'}^{-1}B)=N\left(f(B^{-1}XB^{-1}B{X'}^{-1}B) \right)=N\left(f(X{X'}^{-1})\right)\\
= N\left(f(X'({X'}^{-1}X){X'}^{-1})\right)=
N\left(f({X'}^{-1}X)\right)=d_{N,f}(X',X).
\end{gathered}
\]
The fact that (c2) is also valid is an immediate consequence of the boundedness of $f$ and the equivalence of the norm $N(.)$ to the operator norm.

We next show that (c3) is satisfied when $A$ and $B$ are close enough to each other in the operator norm. 
By the definition of conformal maps we have that the limit
\begin{equation*}
\lim_{z\to 1}\frac{f(z)-f(1)}{z-1}
\end{equation*}
exists and is different from zero.
It is written on p.2 in \cite{special} that this property  guarantees that the condition $\left|f(z^2)\right|\geq K \left|f(z)\right|$ is satisfied automatically with some positive constant $K>1$ for every $z\in\mathbb{T}$ from a neighbourhood of $1$. Clearly, this implies that 
\begin{equation} \label{close}
N(f(C^2))\geq K \cdot N(C)
\end{equation} holds whenever $C\in SO(n)$ is close enough to $I$ in the operator norm. Select $A,B\in SO(n)$, which are close enough to each other in the operator norm, and pick an arbitrary $X\in L_{A,B}$. Since the norm $N(.)$ is equivalent with the operator norm, by the definition of $L_{A,B}$ and the property (f1) of $f$ we obtain that the quantity
\[
N\left(f(X^{-1}B)\right)=d_{N,f}(X,B)=d_{N,f}(A,B)=N\left(f(A^{-1}B)\right)
\]
is small. By the property (f1) we also have that $X^{-1}B$ is close enough to the identity. According to \eqref{close} we have
\[
\begin{gathered}
d_{N,f}(X,BX^{-1}B)=N\left(f\left(X^{-1}(BX^{-1}B)\right)\right)=N\left(f\left((X^{-1}B)^2\right)\right)\geq\\
K \cdot N\left(f\left(X^{-1}B\right)\right)=K\cdot d_{N,f}(X,B).
\end{gathered}
\]
This verifies the property (c3) for close enough $A,B\in SO(n)$.

Relating to the condition (c4), similarly to the argument we have presented concerning the condition (c1), we compute
\[
\begin{gathered}
d_{N,f}(CY^{-1}C',C{Y'}^{-1}C')=N\left(f({C'}^{-1}YC^{-1}C{Y'}^{-1}C')\right)=\\
N\left(f(Y{Y'}^{-1})\right)=N\left(f({Y'}^{-1}Y)\right)=d_{N,f}(Y',Y).
\end{gathered}
\]

Taking all the information what we have into account we conclude that $\phi$ preserves the inverted Jordan triple product for close enough elements.  
In order to complete the proof, \cite[Lemma 7]{hatori} can be applied for obtaining the global inverted Jordan triple product preserver property if it holds for close enough elements. To check that the assumptions of \cite[Lemma 7]{hatori} are satisfied one can argue by literally following the second part of the proof of \cite[Lemma 3.]{abe}. Since $\phi$ is unital, and thus compatible with the inverse operation, we conclude that $\phi$ preserves (globally) the Jordan triple product, as well.
The reader should consult with \cite{abe} and \cite{hatori} for a more detailed argument.
\end{proof}

Now we are in a position to prove our first result.

\begin{proof}[Proof of Theorem~\ref{T:1}]
Let us begin with the sufficiency part. It is not difficult to verify that if $\phi$ is of the form (a) or (b), then it preserves the quantity $d_{N,f}(.,.)$ between the elements of $SO(n)$. So, we are concerned about the case when $\phi$ is of the form (c) or (d). Let us assume that $\phi$ is of the form (c); the case when $\phi$ is of the form (d) can be handled similarly. By Lemma~\ref{L:0} we conclude that the skew-symmetric matrices $BCH(-X,Y)$ and $BCH(-\widetilde{X},\widetilde{Y})$ has the same Youla decomposition and thus there exists an orthogonal matrix $Q$ such that $BCH(-X,Y)=Q\cdot BCH(-\widetilde{X},\widetilde{Y})\cdot Q^{-1}$. This implies that
\[
f\left(\exp(BCH(-X,Y))\right)=Q\cdot f\left(\exp(BCH(-\widetilde{X},\widetilde{Y}))\right)\cdot Q^{-1}.
\]
Since $N$ is unitary invariant we have 
\[
\begin{gathered}
d_{N,f}\left(\phi(\exp(X)),\phi(\exp(Y))\right)=N\left(f(\exp(BCH(-\widetilde{X},\widetilde{Y})))\right)= \\
N\left(f\left(\exp(BCH(-X,Y))\right)\right)
=d_{N,f}\left(\exp(X),\exp(Y)\right).
\end{gathered}
\]
It is now apparent that any map of the form (c) and (d) is a generalized isometry.

As for the necessity, assume that $\phi$ is a generalized isometry with respect to the generalized distance measure $d_{N,f}(.,.)$. 
Then we observe that $\phi(I)^{-1}\phi(.)$ is a map which has the same preserver property as $\phi(.)$ and sends the unit to the unit.
It means that in the sequel without loss of generality we may and do assume that $\phi(I)=I$.
According to Lemma~\ref{L:1} we have that $\phi$ is a unital Jordan triple endomorphism, as well.

We intend to show now that $\phi$ maps a one-parameter subgroup to another one. In order to do so, for a fixed $X\in \mathbb{K}_n(\mathbb{R})$ consider the one-parameter subgroup generated by $X$, that is, $e^{t X}$ where $t\in\mathbb{R}$. We claim $\gamma(t):=\phi\left(e^{t X}\right)$ is a one-parameter subgroup of $SO(n)$, as well.
To verify this property it is sufficient to show that \eqref{subgroup} holds in the case where $t,s\in \mathbb{Q}$. Indeed, then \eqref{subgroup} follows from the facts that the rationals are dense in $\mathbb{R}$ and $\phi$ is continuous. So, consider the numbers $t=p/q$ and $s=r/m$ where $p,q,r$ and $m$ are integers. Then we compute
\[
\begin{gathered}
\gamma(t+s)=\phi\left(e^{(t+s)X}\right)=\phi\left(e^{\left(\frac{pm+rq}{qm}\right)X}\right)=\\
\phi\left(e^{\frac{X}{qm}}\right)^{pm+rq}=\phi\left(e^{\frac{X}{qm}}\right)^{pm}\cdot\phi\left(e^{\frac{X}{qm}}\right)^{rq}=\\
\phi\left(e^{\frac{p}{q}X}\right)\cdot \phi\left(e^{\frac{r}{m}X}\right)=\gamma(t)\cdot\gamma(s)
\end{gathered}
\]
and this verifies our claim. 

Since the exponential map maps from ${\mathbb K}_n({\mathbb R})$ onto $SO(n)$ we obtain that there exists an $Y\in \mathbb{K}_n(\mathbb{R})$, the generator of $\gamma(t)$, such that $\gamma(t)=e^{t Y}$. We constitute a map $h\colon \mathbb{K}_n(\mathbb{R})\to \mathbb{K}_n(\mathbb{R}), X\mapsto Y$ such that $\phi\left(e^{t X}\right)=e^{t h(X)}$ holds for every $t\in\mathbb{R}$. Since $\phi$ respects the generalized distance measure $d_{N,f}$ it is clearly injective. It implies that $h$ is injective, as well. We assert that $h$ is surjective. Indeed, considering $\phi^{-1}$ into the place of $\phi$ by the above argument we conclude that there is an injective map $g:\mathbb{K}_n(\mathbb{R})\to \mathbb{K}_n(\mathbb{R})$ for which $\phi^{-1}\left(e^{t X}\right)=e^{t g(X)}$ holds for every $t\in\mathbb{R}$ and $X\in\mathbb{K}_n(\mathbb{R})$. This gives us that $h(g(X))=X$ holds for every $X\in\mathbb{K}_n(\mathbb{R})$. Hence $h$ is surjective and thus a bijection on $\mathbb{K}_n(\mathbb{R})$.

Next we prove that $h:\mathbb{K}_n(\mathbb{R})\to \mathbb{K}_n(\mathbb{R})$ is a linear isometry with respect to the unitary invariant norm $N(.)$.
Since $f$ is conformal in a neighbourhood of $1$ with the property that $f(1)=0$ we obtain that $f$ has locally a power series expansion of the form 
\[
f(z)=\sum_{k=1}^{\infty} a_k (z-1)^{k}
\]
with $a_1\neq 0$.
Let us consider any skew-symmetric matrices $X,Y\in \mathbb{K}_n(\mathbb{R})$. It is apparent that
\[
e^{-tX}e^{tY}=I+(Y-X)t+\mathcal{O}(t^2)
\]
and thus 
\[
d_{N,f}\left(e^{tX},e^{tY}\right)=N\left(a_1(Y-X)t+\mathcal{O}(t^2)\right).
\]
Similarly, we have
\[
\begin{gathered}
d_{N,f}\left(\phi\left(e^{tX}\right),\phi\left(e^{tY}\right)\right)=d_{N,f}\left(e^{th(X)},e^{th(Y)}\right)\\=N\left(a_1(h(Y)-h(X))t+\mathcal{O}(t^2)\right).
\end{gathered}
\]
Since $\phi$ is a generalized isometry on $SO(n)$ we conclude that
\begin{equation} \label{limit}
N\left(a_1(Y-X)+\mathcal{O}(t)\right)=
N\left(a_1(h(Y)-h(X))+\mathcal{O}(t)\right).
\end{equation} 
Taking the limit $t\downarrow 0$ in \eqref{limit} yields that $h:\mathbb{K}_n(\mathbb{R})\to \mathbb{K}_n(\mathbb{R})$ is a surjective isometry with respect to the metric induced by the unitary invariant norm $N(.)$. 
By the definition of $h$ we also have $h(0)=0$. Consequently, an application of the celebrated Mazur-Ulam theorem yields that $h$ is linear.

The structure of linear isometries on $K_n(\mathbb{R})$ is described in \cite{li, chi} (see also \cite{guralnick} for the details of the proof) with respect to any orthogonal congruence invariant norm which is not a constant multiple of the Frobenius norm. By that result we have that there exist a real number $\eta\in \{-1,1\}$ and an orthogonal matrix $Q\in O(n)$ such that
\begin{itemize}
\item[(aa)] $h(X)=\eta QXQ^{-1}$ for every $X\in \mathbb{K}_n(\mathbb{R})$;
\item[(bb)] $n=4$ and $h(X)=\eta Q\widetilde{X}Q^{-1}$ for every $X\in \mathbb{K}_n(\mathbb{R})$. 
\end{itemize}
From this we infer that
\[
\phi(\exp X)= \exp h(X)= \exp \left(\eta QXQ^{-1}\right) =Q\left( \exp \eta X \right) Q^{-1}
\]
and this results (a) and (b) when $h$ is of the form (aa) and $A=\exp X$. The case (bb) can be handled by a similar way. The proof is complete.
\end{proof}

We continue with the proof of our second theorem.

\begin{proof}[Proof of Theorem~\ref{T:2}]
Parallel with the proof of Theorem~\ref{T:1}, we obtain that
\[
\phi(A)=\exp(h(X)),\qquad X\in\mathbb{K}_3(\mathbb{R})
\]
where $A=\exp X$ and $h:\mathbb{K}_3(\mathbb{R})\to\mathbb{K}_3(\mathbb{R})$ is a linear isometry with respect to the Froebenius norm. 
Apparently, if the eigenvalues of the skew-symmetric matrix $X\in \mathbb{K}_3(\mathbb{R})$ are $\{\mathrm{i}\lambda,-\mathrm{i}\lambda,0 \}$ with some $\lambda \geq 0$, then 
its singular values are $\sigma_1=\sigma_2=\lambda$ and $\sigma_3=0$. Hence we have
\[
\|X\|_F=\sqrt{2}\cdot\|X\|=\sqrt{2}\cdot\sigma_{max}=\lambda.
\]
It means that the Frobenius norm is a $c$-spectral norm on $\mathbb{K}_3(\mathbb{R})$ with $c=(\sqrt{2},0,0)$. Therefore, there exist a real number $\eta\in \{-1,1\}$ and an orthogonal matrix $Q\in O(3)$ such that $h$ is of the form (aa) (see \cite[Theorem 4.2]{chi}). Now, just as at the end of the proof of Theorem~\ref{T:1}, we conclude that $\phi$ is of the desired form.
\end{proof}

\subsection*{Acknowledgements}

The author thanks the anonymous referee for pointing out some flow in the presentation of the paper and his/her valuable comments and suggestions.

This work was partially supported by the ``Lend\" ulet'' Program (LP2012-46/2012) of the Hungarian Academy of Sciences and the National Research, Development and Innovation Office -- NKFIH Reg. No. K115383.  

The author also thanks Prof. Lajos Moln\'ar for encouragement.

\end{document}